\documentclass{amsart}
\usepackage{color}
\usepackage[margin=1in]{geometry}
\usepackage{graphicx,amsthm,amssymb,tikz,bbold,comment}
\usepackage{cite}
\usetikzlibrary{cd}

\usepackage{xcolor}
\colorlet{darkgreen}{green!40!black}
\usepackage{ifpdf}
\ifpdf \usepackage[pdftex]{hyperref}
\else \usepackage[ps2pdf]{hyperref} 
\usepackage{breakurl}
\fi
\hypersetup{colorlinks=true,urlcolor=blue,citecolor=darkgreen,
	linkcolor=darkgreen,linktocpage,bookmarksnumbered,unicode}
\newcommand{\R}{\mathbb{R}}\newcommand{\Z}{\mathbb{Z}}\newcommand{\Q}{\mathbb{Q}}\newcommand{\C}{\mathbb{C}}
\newcommand{\e}{\varepsilon}
\newcommand{\Ker}{\textnormal{Ker}}

\newcommand{\tn}[1]{\textnormal{#1}}

\newcommand{\ov}[1]{\overline{#1}}
\renewcommand{\k}[1]{\langle #1 \rangle}
\renewcommand{\sb}{\subseteq}
\newcommand{\ra}{\rightarrow}
\newcommand{\xra}[1]{\xrightarrow{#1}}
\newcommand{\MTU}{\mathbf{MTU}}\newcommand{\MU}{\mathbf{MU}}\newcommand{\Hom}{\mathbf{Hom}}

\newtheorem{remark}{Remark}[section]

\newtheorem{definition}{Definition}[section]

\newtheorem{theorem}{Theorem}[section]
\newtheorem{proposition}[theorem]{Proposition}

\newtheorem{corollary}[theorem]{Corollary}
\newtheorem{lemma}[theorem]{Lemma}
\begin{document}
\title{Cobordism Obstructions to Complex Sections}
\author{Dennis Nguyen}
\email{dpn@uoregon.edu}
\address{ Department of Mathematics, University of Oregon, Fenton Hall, 1021 E 13th Ave, Eugene, OR, 97403 }


	\begin{abstract}
		A notion of vector field cobordism for oriented manifolds was defined by B\"okstedt and Svane. We extend this notion to define complex section cobordism for almost complex manifolds. We then determine the complex section cobordism groups and a relevant cobordism category. We describe an obstruction which tells us when a cobordism class contains a manifold, which can be equipped with $r$ linearly independent complex sections, in terms of the Chern numbers. Finally, we show that this obstruction vanishes for certain multiplicative generators in the complex cobordism ring.
	\end{abstract}
\maketitle
	\section{Introduction}
	\subsection{Summary of Results}
	
	There is a classical problem to determine whether a manifold admits $r$
        linearly independent tangent vector fields. In the case
        of one everywhere non-zero vector field, this problem was solved by Hopf \cite{hopf} and the
        obstruction is the Euler characteristic of the manifold. B\"okstedt, Dupont and Svane \cite{bokstedt} approached this
        problem by instead determining the obstruction to finding a cobordant manifold with $r$ vector fields. For small $r$, they were able to solve this problem in
        \cite{bokstedt}. We extend their results by looking at
        obstructions to finding linearly independent complex sections of the tangent bundle of almost complex manifolds. In this case, we are able to describe the rational obstruction for almost complex manifolds. This obstruction is given in terms of Chern characteristic numbers. Our main result is the following sufficient condition for finding $r$ complex sections:
	
	\begin{theorem}\label{ratobsintro}
		Let $M^{2d}$ be a $d$-dimensional almost-complex manifold. Then there exists integer $c$ such that the cobordism class $c[M^{2d}]$ contains a manifold $N^{2d}$ with $r$ complex sections on the tangent bundle $TN$ if and only if $s_\omega(M^{2d})=0$ for all $\omega$ of length greater than $d-r$.
	\end{theorem}
	
	The characteristic classes $s_\omega$ and the proof of this Theorem are given in section 4. B\"okstedt and Svane defined vector field cobordism in \cite{bokstedt2}. We extend their definition to \textit{complex section cobordism}, which we split into even and odd dimensional cases. 
	
	\begin{definition}[Definition \ref{evencsdef}] \label{evencsdefintro}
		Two $2d$ dimensional manifolds $M$ and $N$ with almost complex structure and $r$ complex sections on $TM$ and $TN$ are defined to be complex section cobordant, if there is a cobordism $W$ with boundary $M\cup \overline{N}$ such that $TW\oplus \R$ has complex structure and $r$ linearly independent complex sections compatible with the structures on $TM\oplus \C$ and $TN\oplus \C$. ($\overline{N}$ is  $N$ with the opposite orientation.) The even dimensional complex section cobordism groups are the equivalence classes under this relation. 
	\end{definition}
	
 \begin{definition}[Definition \ref{oddcsdef}] \label{oddcsdefintro}
	Two $2d-1$ dimensional manifolds $M$ and $N$ with complex structure and $r$ linearly independent complex sections on $TM\oplus \R$ and $TN\oplus \R$ are defined to be complex section cobordant, if there is a cobordism $W$ with boundary $M\cup \overline{N}$ such that $TW$ has complex structure and $r$ linearly independent complex sections compatible with the structures on $TM\oplus \R$ and $TN\oplus \R$. The odd dimensional complex section cobordism groups are the equivalence classes under this relation.
\end{definition}

	Note that a priori, even for $r=0$, these definitions are not the same as the definition of the classical complex cobordism groups, since the classical complex cobordism groups allow stabilization of arbitrarily high dimension.
	
	 We can describe how many distinct ways we can equip a cobordant manifold with the complex sections. This is equivalent to describing the kernel of the forgetful map from the complex section cobordism group to the complex cobordism groups. In section 4, we prove that this kernel is finite and is one of the homotopy groups of a particular spectrum. We note that this kernel is not necessarily finite in the oriented case. In particular, the forgetful map from the Reinhart cobordism group \cite{reinhart} to the oriented cobordism group is $\Z$ in even dimensions.
	
	\begin{theorem}[Corollary \ref{T2body}]\label{T2}
		Let $[M]\in \Omega_{2d}^U$, be such that $M$ can be equipped with $r$ linearly independent complex sections on $TM$. Then, there are only finitely many ways to equip a manifold $N\in [M]$ with $r$ linearly independent complex sections up to complex section cobordism. 
	\end{theorem}
	The ways of equipping a manifold with $r$ complex sections are indexed by a group which is described in section 4. In odd dimensions, it is well known that the complex cobordism group is zero, thus the odd dimensional complex section cobordism group parameterizes all the ways of equipping a manifold in the unique cobordism class with $r$ complex sections. The odd dimensional complex section cobordism group is in general more difficult to compute than the even dimensional complex section cobordism group, however we note that:
	
	\begin{theorem}[Proposition \ref{T3body}]\label{T3}
		The odd dimensional complex section cobordism group is finite.
	\end{theorem}

	Our final question is finding multiplicative generators of the complex cobordism ring which can be equipped with $r$ complex sections. While we were unable to give an integral result stating which dimensions will have a generator with the obstruction vanishing, we showed that they could be found in rational cobordism.

\begin{theorem}[Theorem \ref{T4body}]\label{T4}
	Let $r< d$. There exists a manifold $M^{2d}$ in $\Omega_{2d}^U$ which can be equipped with $r$ linearly independent complex sections on $TM$ and whose image in $\Omega_{2d}^U\otimes \Q$ is a multiplicative generator.
\end{theorem}

\subsection{Plan of the paper}

We start in section 2.1 by constructing the spectra $\MTU(d)$, and $\MTU(d,r)$. This constructions closely mirrors work done in \cite{bokstedt}. We present two constructions of the spectrum $\MTU(d)$. One of which is useful for calculating the cohomology and the other has clearer geometric features. Section 2.2 is devoted to compute the cohomology of those spaces and construct the colimit spectrum $\overline{\MTU}(d)$ based on a stabilization map. Section 3 describes the complex section cobordism category and provides the geometric definitions of the complex section cobordism groups. We use the results of Galatius, Tillmann, Madsen and Weiss \cite{galatius} to connect these groups to the homotopy groups of the spectrum $\MTU(d)$. This section is divided into two parts, one describing the odd dimensional and one describing the even dimensional case. We prove the main results of the paper in section 4 using the homotopy exact sequences.

{\bf Acknowledgments:} I would like to thank all the people who have helped me as I have completed this paper. First, I would like to thank the University of Oregon Department of Mathematics for supporting me during this work. I most especially thank my advisor, Boris Botvinnik, for the time he spent guiding and advising me throughout this project. I am also indebted to S\o ren Galatius for his valuable feedback on my results. Finally, I would like to thank the Young Topologists Meeting and the University of Copenhagen for giving me an opportunity to share an early version of this project.
\section{Construction of Spectra}
\subsection{The spectra MTU(d) and MTU(d,r)}
Here we introduce the spectra $\MTU(d)$ and $\MTU(d,r)$. The homotopy groups of these spectra are the natural objects of study. The real case is described in \cite{bokstedt} and \cite{galatius}. It should be noted that there is some disagreement about the proper indexing of these spectra in \cite{bokstedt} and \cite{galatius}. We follow the convention given in B\"okstedt, Dupont and Svane. \cite{bokstedt} Denote the complex Grassmannian of $d$ dimensional complex subspaces of
$\C^{d+n}$ by $G_\C(d,n)$. Let $U_{\C, d,n}\ra G_\C(d,n)$ be the tautological $d$
dimensional complex vector bundle and let $U_{\C, d,n}^\perp \ra
G_\C(d,n)$ be the $n$ dimensional complement of $U_{\C, d,n}$. 
	\begin{definition}
		Define $\MTU(d)$ to be the spectrum whose $2n$-th
                space is: \[\MTU(d)_{2n}= Th(U_{\C, d,n}^\perp)\]
	\end{definition}
There exists a
canonical map $G_\C(d,n)\ra G_\C(d,n+1)$ defined using the
composition: \[\C^d\ra \C^{d+n} \hookrightarrow \C \oplus \C^{d+n} \cong
\C^{d+n+1}\]
 The restriction of the bundle $U_{\C, d,n+1}^\perp$ to
$G_\C(d,n)$ under this map is $U_{\C, d,n}^\perp\oplus \C $. (Here $\C$ denotes a one dimensional complex trivial bundle.) In other words, there is a bundle map $U_{\C, d,n}^\perp\oplus \C\ra U_{\C,
  d,n+1}^\perp$ covering the map $G(d,n)\ra G(d,n+1)$. This map
induces a map of Thom spaces and the following composition gives the spectrum map $\Sigma^2 \MTU(d)_{2n}\ra \MTU(d)_{2n+2}$: 
\[\Sigma^2(Th(U_{\C,d,n}^\perp))\cong S^2\wedge Th(U_{\C, d,n}^\perp) \cong Th(\C \oplus
U_{\C, d,n}^\perp) \ra Th(U_{\C, d,n+1}^\perp)\]
There is also a map $G(d-r,n)\ra G(d,n)$ which takes a $(d-r)$-complex plane $P\subset \C^{d-r+n}$ to the $d$-plane $P\oplus \C^r \subset
\C^{d-r+n} \oplus \C^r$. Under this map, the pullback of $U_{\C, d,n}^\perp$
is $U_{\C, d-r,n}^\perp$. Thus, there are maps of Thom spaces
$\MTU(d-r)_{2n}\ra \MTU(d)_{2n}$. Since this map commutes with the
spectrum map, it defines a map of spectra:
\begin{equation}\label{s1-eq1}
\MTU(d-r)\ra \MTU(d)
\end{equation}  
\begin{definition}
	Let $\MTU(d,r)$ be the cofiber of the map (\ref{s1-eq1}).
\end{definition}
 We note that, having defined $\MTU(d,r)$, we immediately get a cofibration for $k\leq d-r$.
\begin{equation}
	\MTU(d-r,k)\ra \MTU(d,r+k) \ra \MTU(d,r)
\end{equation}
This cofibration reduces to the definition when $k=d-r$.

 	There is a second construction of spectrum $\MTU(d)$ due to \cite{bokstedt} which will be used in section 3 to study the complex section cobordism category.
 	
 	For any $d$-dimensional complex fiber bundle $E\ra X$ equipped with a Hermitian inner product, there is a complex frame bundle $V_{\C,r}(E)\ra X$ with fiber $V_{\C,d,r}$. The space  $V_{\C,d,r}$ is the frame manifold of $r$ ordered complex frames in $\C^d$. There is a related bundle $W_{\C,r}(E)\ra X$, whose fiber is the space of ordered $r$-tuples in $\C^d$ which are (hermitian) orthogonal and of the same length which may be between $0$ and $1$. The fiber $W_{\C,d,r}$ is the cone over $V_{\C,d,r}$, by construction.
 	
 	Now consider the specific case where the bundle is $U_{\C,d,n}\ra G_\C(d,n)$.
 	Elements of $V_{\C,r} (U_{\C,d,n})$ consist of a complex $d$ dimensional plane $P\subset \C^{d+n}$ along with an $r$ complex-frame in that plane. There is a map \[\eta^r:G_\C(d-r,n)\ra V_{\C,r} (U_{\C,d,n})\] which takes a $(d-r)$-plane $P\subset \C^{d-r+n}$ to $P\oplus \C^r \subset \C^{d-r+n}\oplus \C^r$ with $r$ frame consisting of the standard basis vectors of $\C^r$.
 	
 	We extend $\eta^r$ to a section $\eta:G_\C(d,n)\ra W_{\C,2r} (U_{\C,d,n})$ as follows. Let $\e>0$ be small and $N$ be the open neighborhood of $G_\C(d-r,n)$ of planes in $G_\C(d,n)$ which differ from a plane in $G_\C(d-r,n)$ by a rotation by angle less than $\e$.
 	
 	Let $P$ be a $d$-dimensional complex plane in $\overline{N}$, the closure of $N$. If $\e$ is sufficiently small, there is a unique shortest smooth rotation $A\in SU(d+n)$ from $P$ to a unique $P_0\in G_\C(d-r,n)$. (By shortest, we mean smallest angle.) Then choose a frame by applying the rotation $A^{-1}$ to $\eta^r(P_0)$. This will rotate the standard frame in $\C^r$ to be in $P$. If $\e$ is small, this map will be continuous. Define the section $\eta$ on $\ov{N}$ by taking plane $P$ and equipping it with this frame produced by the above rotation, scaling the frame by $(\e-\theta)/\e$ where $\theta$ is the angle of the rotation. Note that if $P\in G_\C(d-r,n)$, this process will give us the standard $r$ frame in $\C^r$. Moreover on the boundary of $N$, all $d$-planes are equipped by this map with the zero frame. (The zero frame is the frame with all vectors being zero vectors; it is the cone point of the fiber.) Thus we can extend this map to all of $G_\C(d,n)$ by mapping planes outside of $N$ to the plane equipped with the zero frame. This section $\eta: G_\C(d,n) \ra W_{\C,r} (U_{\C,d,n})$ is the right vertical map in the left diagram below.

 			\begin{center}
 				\begin{tikzcd}[column sep=large, row sep=huge]
 					V_{\C,r} (U_{\C,d,n})
 					\arrow[r]  
 					\arrow[d, leftarrow, "\eta^r"]& W_{\C,r} (U_{\C,d,n}) \arrow[d, leftarrow, "\eta"] \\
 					G_\C(d-r,n) \arrow[r] 	& G_\C(d,n)
 				\end{tikzcd}
 				\begin{tikzcd}[column sep=large, row sep=huge]
 					p_{V_{\C,r}}^* U_{\C,d,n}^\perp
 					\arrow[r]  
 					\arrow[d, leftarrow, "\eta^r"]& p_{W_{\C,r}}^* U_{\C,d,n}^\perp \arrow[d, leftarrow, "\eta"] \\
 					U_{\C,d-r,n}^\perp \arrow[r] 	& U_{\C,d,n}^\perp
 				\end{tikzcd}
 			\end{center}

 	The right commutative diagram lies over the first one where the maps $p_{V_{\C,r}}: V_{\C,r} (U_{\C,d,n})\ra G_\C(d,n)$ and $p_{W_{\C,r}}: W_{\C,r} (U_{\C,d,n})\ra G_\C(d,n)$ are the corresponding projections. Then we  have the Thom spaces $Th(p_{V_{\C,r}}^* U_{\C,d,n}^\perp)$ and $Th(p_{W_{\C,r}}^* U_{\C,d,n}^\perp)$ which we form into spectra $\MTU(d)_{V_r}$ and $\MTU(d)_{W_r}$ respectively. The maps $\eta^r$ and $\eta$ determine the following maps of spectra: $\MTU(d)_{V_r}\ra \MTU(d-r)$ and  $\MTU(d)_{W_r}\ra \MTU(d)$. The top maps in the above diagram induce a map of spaces $Th(p_{V_{\C,r}}^* U_{\C,d,n}^\perp) \ra Th(p_{W_{\C,r}}^* U_{\C,d,n}^\perp)$ and a map of spectra, $\MTU(d)_{V_r}\ra \MTU(d)_{W_r}$. From this map we form a cofiber which we call $\MTU'(d,r)$.
 	
 	\begin{proposition}\label{heq-1}
	In the below commutative diagram, all vertical maps are homotopy equivalences.
	
	\begin{center}
		\begin{tikzcd}[column sep=large, row sep=huge]
			\MTU(d-r)
			\arrow[r]  
			\arrow[d, leftarrow, "\eta^r"]& \MTU(d)	\arrow[r]  \arrow[d, leftarrow, "\eta"] & \MTU(d,r) \arrow[d, leftarrow, "\overline{\eta}"]\\
			\MTU(d)_{V_r} \arrow[r] 	& \MTU(d)_{W_r}\arrow[r] & \MTU'(d,r)
		\end{tikzcd}
	\end{center}
\end{proposition}
\begin{proof}
	We know the section $\eta:G_\C(d,n) \ra W_{\C,r}(U_{\C,d,n})$ is a homotopy inverse of $p_{W_{\C,r}}$ because $W_{\C,r}(U_{\C,d,n})$ has contractible fibers. There is a fiber bundle \[V_{\C,r}(U_{\C,d,n}) \ra V_{\C,d+n,r}\]
	The total space $V_{\C,r}(U_{\C,d,n})$ consists of a $d$ dimensional plane in $\C^{d+n}$ and an $r$-frame in that plane. The projection forgets the plane leaving only the frame. For a given frame in $V_{\C,d+n,r}$, the fiber consists of any $d$ dimensional complex plane containing the frame. This is equivalent to choosing a $(d-r)$-plane in the orthogonal complement of the frame. The orthogonal complement is $\C^{d+n-r}$, so the fiber is $G_\C(d-r,n)$. 
	The fiber inclusion takes a complex plane $P^{d-r} \sb \C^{n+d-r}$ to the plane $P^{d-r} \oplus \C^r \sb \C^{n+d}$ equipped with the frame of the $r$ standard basis vectors in $\C^r$. This is the map $\eta^r$.
	
	We know the base of the fibration, $V_{\C, d+n,r}$, is $(2n+2d-2r-1)$ connected and thus the pair of spaces $(V_{\C,r}(U_{\C,d,n}),G_\C(d-r,n))$ has the same connectivity. By the Thom isomorphism theorem, we find the pair $(Th(p_{V_{\C,r}}^* U_{\C,d,n}^\perp), Th(U_{\C,d-r,n}^\perp) )$ is $(4n+2d-2r-1)$ connected.
	
	Letting $n\ra \infty$, we see that $\eta^r: \MTU(d)_{V_r}\ra \MTU(d-r)$ is a homotopy equivalence. By the five lemma, the right vertical map is also a homotopy equivalence.
\end{proof}

\subsection{Cohomology of the spectra}

We next compute the cohomology of the spectra $\MTU(d)$ and $\MTU(d,r)$. For the rest of this section, all cohomology will be assumed to have $\Z$ coefficients.

\begin{proposition}\label{cohMTU}
	There is an isomorphism:
	\[\phi:\Z[c_1,c_2,...,c_d]\cong H^*(BU(d);\Z)\ra H^*(\MTU(d);\Z)\]
	
	where $H^*(\MTU(d);\Z)$ is considered as a $H^*(BU(d);\Z)$ module. 
\end{proposition}
\begin{proof}
	There is a Thom class $\ov{u}_{d,n}\in H^{2n}(Th(U^\perp_{\C,d,n}))$ associated to the bundle $U^\perp_{\C,d,n}$. In the $2n$-th space, this defines a Thom isomorphism $H^*(G_{\C,d,n}) \ra H^{*+2n}(Th(U^\perp_{\C,d,n}))$. In the limit we get a stable Thom class $\ov{u}_d \in H^0(\MTU(d))$. The Thom isomorphism theorem states that $H^*(\MTU(d))$ is the rank 1 free module over $H^*(BU(d))$ generated by the Thom class.
\end{proof}

\begin{theorem}\label{map}
	The map $H^*(\MTU(d,r))\ra H^*(\MTU(d))$ is injective with image the $H^*(BU(d))$ module generated by $\phi(c_{d-r+1}),...,\phi(c_d)$.
	
\end{theorem}

\begin{proof}
	Observe the following commutative diagram:
	
	\begin{center}
		\begin{tikzpicture}
			\node (v4) at (-2,-0.5) {$H^k(G_\C(d, n))$};
			\node (v5) at (1.7,-0.5) {$H^k(G_\C(d - r, n))$};
			\node (v2) at (-2,1.5) {$\tilde{H}^{k+2n}(Th(U^\perp_{\C, d,n}))$};
			\node (v6) at (-6.5,-0.5) {$H^k(G_\C(d, n), G_\C(d - r, n))$ };
			\node (v1) at (-6.5,1.5) {$H^{k+2n}(Th(U^\perp_{\C, d,n}, Th(U^\perp_{\C, d-r,n}))$};
			\node (v3) at (1.7,1.5) {$\tilde{H}^{k+2n}(Th(U^\perp_{\C, d-r,n}))$};
			\draw [->] (v1) edge (v2);
			\draw [->] (v4) edge (v5);
			\draw [->] (v6) edge (v4);
			\draw [->] (v4) edge (v2);
			\draw [->] (v5) edge (v3);
			\draw [->] (v6) edge (v1);
			\draw [->] (v2) edge (v3);
		\end{tikzpicture}
	\end{center}

	The horizontal maps come from the exact sequence in cohomology. As $n\ra \infty$ the bottom right map becomes $\Z[c_1,..c_d]\cong H^k(BU(d))\ra H^k(BU(d-r))\cong \Z[c_1,...,c_{d-r}]$. This map is a surjection mapping $c_i\ra c_i$ for $i\leq d-r$. So  $H^*(BU(d),BU(d-r))$ is the kernel, namely the $BU(d)$ submodule generated by $c_{d-r+1},...,c_d$.
	
	The vertical maps are the Thom isomorphisms. Thus we conclude that the map $H^*(\MTU(d,r))\ra H^*(\MTU(d))$ is injective with image the $H^*(BU(d))$ submodule generated by $\phi(c_{d-r+1}),...,\phi(c_d)$. 
\end{proof}	
\begin{corollary}\label{connected}
	The spectrum $\MTU(d,r)$ is $(2(d-r)+1)$ connected.
\end{corollary}

As a corollary, there is a stabilization isomorphism for $\MTU(d,r)$
\begin{theorem} \label{stab}
	$\pi_q(\MTU(d,r))\cong \pi_q(\MTU(d+k,r+k))$ for $q\leq 2d$. 
\end{theorem}
\begin{proof}
	There is a homotopy exact sequence:
	\[
		\pi_{q+1}(\MTU(d+k,k))\ra \pi_q(\MTU(d,r))\\ \ra \pi_q(\MTU(d+k,r+k)) \ra \pi_q (\MTU(d+k,k))\]

	By the Hurewicz theorem, $\MTU(d+k,k)$ is $2(d+k-k)+1$ connected. So there is an isomorphism, $\pi_q(\MTU(d,r))\cong \pi_q(\MTU(d+k,r+k))$ for $q\leq 2d$. 
\end{proof}

We will connect the homotopy groups of $\MTU(d)$ to the homotopy groups of $\MU$, which are the complex cobordism groups. There is a sequence of spectra:
\[ \MTU(d)\ra \MTU(d+1) \ra \MTU(d+2)\ra ...\]

Let us call the colimit of this spectrum $\MTU$; this spectrum is the Thom spectrum of the bundle $U^\perp_\C$ over $BU$. There is an inversion map $BU\ra BU$ which takes a vector space to its complement. This map is covered by a bundle map $U^\perp_\C\ra U_\C$ which defines a homotopy equivalence $\MTU\ra \MU$. (The real version of this construction is described in \cite{ebert} among others.) In particular, there is a well defined map $\MTU(d)\ra \MU$. The computation in Theorem \ref{map} implies that the map $\MTU(d)\ra \MU$ is $2d+1$ connected. (The map on cohomology is the quotient $\Z[c_1,..,c_d,...]\ra \Z[c_1,...,c_d]$.) We extend the results on $\MTU(d,r)$ to a new spectrum, which will be more convenient for the calculations in section 4.

\begin{definition}
	Let $\overline{\MTU}(d)$ be the cofiber of the map $\MTU(d)\ra \MU$.
\end{definition}

\begin{proposition}\label{cohMTUbar}
	The map $H^*(\overline{\MTU}(d)),\Z)\ra H^*(\MU,\Z)$ is injective. Specifically $H^*(\overline{\MTU}(d)),\Z)$ is torsion-free and has non-zero cohomology only in even degrees.
\end{proposition}
\begin{proof}
	There is an exact sequence: \[H^{q-1}(\MU,\Z)\ra H^{q-1}(\MTU(d))\ra H^q(\overline{\MTU}(d)),\Z)\ra H^q(\MU,\Z)\ra H^q(\MTU(d))\]
	
	If $q$ is even, then $H^{q-1}(\MTU(d))$ is 0 so the map is injective. If $q$ is odd, then the induced map $H^{q-1}(\MU,\Z)\ra H^{q-1}(\MTU(d))$ coincides with $\Z[c_1,c_2,...]\ra \Z[c_1,...,c_d]$ and is surjective. It immediately follows that $H^q(\overline{\MTU}(d)),\Z)\ra H^q(\MU,\Z)$ is injective.
\end{proof}

There is an alternative interpretation of $\overline{\MTU}(d)$ as the colimit of the sequence of spectra: \[ \MTU(d+1,1)\ra \MTU(d+2,2) \ra ...\]
The following commutative diagram makes it clear that these definitions are the same.
\begin{center}
	
	\begin{tikzpicture}
		\node (v1) at (-3,-1) {$\MTU(d)$};
		\node (v2) at (0,-1) {$\MTU(d+1)$};
		\node (v3) at (3,-1) {$\MTU(d+1,1)$};
		\draw [->] (v1) edge (v2);
		\draw [->] (v2) edge (v3);
		\node (v5) at (0,0) {$\MTU(d+2)$};
		\node (v9) at (3,0) {$\MTU(d+2,2)$};
		\node (v4) at (0,1) {\vdots};
		\node (v8) at (3,1) {\vdots};
		
		\node (v6) at (0,2) {$\MU$};
		\node (v7) at (3,2) {$\overline{\MTU}(d)$};
		\draw [->] (v1) edge (v6);
		\draw [->] (v1) edge (v5);
		\draw [->] (v6) edge (v7);
		\draw [->] (v8) edge (v7);
		\draw [->] (v9) edge (v8);
		\draw [->] (v5) edge (v4);
		\draw [->] (v4) edge (v6);
		\draw [->] (v3) edge (v9);
		\draw [->] (v2) edge (v5);
		\draw [->] (v5) edge (v9);
	\end{tikzpicture}
\end{center}
As a corollary of Theorem \ref{stab}, \begin{corollary}\label{stabcor} There is an isomorphism for all $d$ and $r$:
	\[\pi_{2d+1}(\MTU(d+1,r+1))\cong \pi_{2d+1}(\overline{\MTU}(d-r))\]
\end{corollary}

This corollary will be helpful when we want to compute the homotopy groups $\pi_{2d+1}(\overline{\MTU}(d-r))$. As a further consequence of Corollary \ref{connected}, we conclude the following lemma.

\begin{lemma} \label{iso}
	There is an isomorphism $\pi_{q}(\MTU(d))\cong \pi_{q}(\MU)$ for $q\leq 2d$. 
\end{lemma}

\section{The cobordism category}
In order to provide a geometric interpretation for the spectra constructed in the previous section, we discuss cobordism categories. These categories were originally studied by \cite{galatius} and \cite{bokstedt2}. Here we specify what happens in the complex case. We will need to reference the spectrum $\mathbf{MTO}(d)$, which is the real analogue for $\mathbf{MTU}(d)$. For more details on its specific construction see \cite{galatius} and \cite{bokstedt2}. In order to describe the complex section cobordism category, we consider the even and odd dimensional cases separately. We start with the general cobordism category with tangential structure $\theta$.

\begin{definition}\label{cobcat}
	Let $\theta:X\ra BO(d)$ be a fibration. The cobordism category $\mathcal{C}^\theta_{d,n+d}$ has as objects $(d-1)$ dimensional manifolds $M\subset\ (-1,1)^{n+d-1} \subset \R^{n+d-1}$ with $M$ closed in $\R^{n+d-1}$ without boundary along with the data of a chosen lift of the classifying map $\xi:M\ra G(d-1,n) \ra G(d,n)$ to $\theta^*(G(d,n))$. The space of morphisms from $M_0\ra M_1$ is the disjoint union of the identity morphism along with pairs $(W,a)$ with the following properties. $W\sb(-1,1)^{n+d-1}\times \R \subset \R^{n+d-1} \times \R $ is a manifold of dimension $d$, which is closed in $\R^{n+d}$, $a\in (0,\infty)$ and for some $\e>0$:
	\[W \cap (\R^{n+d-1}\times (-\infty,\e)) =M_0\times (-\infty,\e)\]
	\[W \cap (\R^{n+d-1}\times (a-\e,\infty)) =M_1\times (a-\e,\infty)\]
	Additionally $W$ is equipped with a chosen lift of the classifying map $W\ra G(d,n)$ to $\theta^*(G(d,n))$ which is compatible with the structures on the cobordant manifolds.
\end{definition}

In this paper, we will take the limit as $n\ra \infty$ of the categories $\mathcal{C}^\theta_{d,n+d}$ and write it as $\mathcal{C}^\theta_{d}$. Its classifying space will be denoted by $BC^\theta_{d}$. In the complex case $X=BU(d)$ and $\theta: BU(d)\ra BO(2d)$. So the $U(d)$ cobordism category has as objects $2d-1$ dimensional manifolds $M$ with $U(d)$ structure on $TM\oplus \R$ and has as morphisms $2d$ dimensional cobordisms with $U(d)$ structure.

\subsection{The even dimensional complex section cobordism category}

We describe next the complex section cobordism category, which will be obtained by additionally requiring the objects and morphisms to be equipped with linearly independent complex sections. In particular, this corresponds to Definition \ref{cobcat}, with the tangential structure given by the map $\theta_{\C,r}: V_{\C,r}(U_{\C,d})\ra BO(2d)$  which is defined by the following composition:

\begin{center}
	\begin{tikzcd}[column sep=large, row sep=huge]
		V_{\C,r}(U_{\C,d}) \arrow[r]\arrow[d, "i_r"] & V_{r}(U_{2d}) \arrow[d, "i_r"]\\
		BU(d)	\arrow[r]  & BO(2d)
	\end{tikzcd}
\end{center} 

The map $ V_{\C,r}(U_{\C,d}) \ra V_{r}(U_{2d})$ is the natural map which forgets the complex structure. (Note that this diagram is not a pull back square.)
The category $\mathcal{C}^{\theta_{\C,r}}_{2d}$ with this tangential structure has as objects $2d-1$ dimensional manifolds $M$ with a complex structure on the bundle $TM\oplus \R$ and $r$ linearly independent complex sections. A morphism $W: M \ra N$ is a $2d$ dimensional almost complex cobordism equipped with $r$ linearly independent complex sections such that the structures are compatible as given in the definition below. We define two manifolds to be complex section cobordant if there is a morphism between them in this category.

\begin{definition}\label{oddcsdef}
	Two $2d-1$ dimensional manifolds $M$ and $N$ with complex structure and $r$ linearly independent complex sections on $TM\oplus \R$ and $TN\oplus \R$ are defined to be complex section cobordant, if there is a cobordism $W$ with boundary $M\cup \overline{N}$ such that $TW$ has complex structure and $r$ linearly independent complex sections compatible with the structures on $TM\oplus \R$ and $TN\oplus \R$. 
\end{definition}

\begin{proposition}\label{oddeqrel}
	The relation defined in Definition \ref{oddcsdef} is an equivalence relation. 
\end{proposition}
\begin{proof}
	This relation is obviously reflexive and transitive. It remains to show that it is symmetric. Suppose $M$ and $N$ are $2d-1$ dimensional manifolds with complex structure and $r$ linearly independent complex sections on $TM\oplus \R$ and $TN\oplus \R$ and suppose that there is a cobordism $W$ with boundary $M\cup \overline{N}$ such that $TW$ has complex structure and $r$ linearly independent complex sections compatible with the structures on $TM\oplus \R$ and $TN\oplus \R$. 
	
	We may reverse $W$ and get a complex cobordism between $N$ and $M$, which I will call $\tilde{W}$. However the complex sections will be reversed on the boundary. We will correct this reversal by showing there exists a cobordism reversing the sections. Let $\nu: M \ra V_{\C,r}(TM\oplus \R)$ be a map representing complex sections and $-\nu$ be their reverse. Define a section $\tilde{\nu}: M\times [0,1] \ra V_{\C,r}(TM\times \R)\cong  V_{\C,r}(TM\oplus \R) \times [0,1] $ by the formula $\tilde{\nu}(x,t)= ( e^{\pi it}\nu(x),t)$. This section restricts to $\nu$ on the incoming boundary and $-\nu$ on the outgoing boundary. We may perform a similar construction on $N$. Composing these cobordisms with $\tilde{W}$, we get a cobordism from $N$ to $M$ with the correct complex section structure.
\end{proof}

\begin{definition}
	The odd dimensional complex section cobordism groups are the equivalence classes under the above relation.
\end{definition}
We can interpret the complex section cobordism groups as the connected components of the classifying space $B\mathcal{C}^{\theta_{\C,r}}_{2d}$.
\begin{proposition}
	The connected components of $B\mathcal{C}^{\theta_{\C,r}}_{2d}$ are the equivalence classes under complex section cobordism.
\end{proposition}
\begin{proof}
	If two manifolds are equivalent, then there is a morphism in $\mathcal{C}^{\theta_{\C,r}}_{2d}$ connecting them and so they are in the same connected component of $B\mathcal{C}^{\theta_{\C,r}}_{2d}$. If two manifolds are in the same connected component of $B\mathcal{C}^{\theta_{\C,r}}_{2d}$, then there is a zig zag of morphisms connecting them by Theorem 3.4 of \cite{bokstedt2}. By Proposition \ref{oddeqrel}, this means that the two manifolds are complex section cobordant.
\end{proof}

\begin{remark}
	Note that we index the cobordism category by the dimension of the cobordisms but we index the cobordism groups by the dimension of the object manifolds.
\end{remark}

Next we modify the results from \cite{bokstedt2}, specializing to the complex case in order to connect these groups to the Thom spectra constructed in Section 2. This primarily relies on the calculation of the homotopy type of the cobordism category in \cite{galatius}.

For any structure $\theta: X \ra BO(d)$, we define $\theta^* \mathbf{MTO}(d)$ to be the spectrum whose $n$-th space is given by $Th(\theta^* U_{d,n}^\perp)$. In particular, for $\theta: BU(d)\ra BO(2d)$, we observe that $\theta^*(\mathbf{MTO}(2d))=\MTU(d)$. For the structure $p_{V_{\C,r}}: V_{\C,r}(U_{\C,2d}) \ra BU(d)$, we observe that $p_{V_{\C,r}}^* \MTU(d)$ is, by construction, $\MTU(d)_{V_r}$. We proved in Theorem \ref{heq-1} that $\MTU(d)_{V_r} \cong \MTU(d-r)$. Note that $\theta_{\C,r}=\theta \circ p_{V_{\C,r}}$ where the map $\theta_{\C,r}$ is as above. The next theorem from \cite{galatius} and its corollary will allow us to study and compute the complex section cobordism groups by studying the homotopy groups of the Thom spectrum $\MTU(d)$. 
\begin{theorem}
	The spaces $B\mathcal{C}^{\theta_{\C,r}}_{2d}$ and $\Omega^{\infty+2d-1} \MTU(d-r)$ are weakly homotopy equivalent.
\end{theorem}
\begin{proof}
	In \cite{galatius}, it was shown that there is a weak homotopy equivalence \[B\mathcal{C}^{\theta_{\C,r}}_{2d}\ra \Omega^{\infty+2d-1} \theta^*_{\C,r} \mathbf{MTO}(2d)\]
	
	By construction, $\theta^*_{\C,r} \mathbf{MTO}(d) \cong \MTU(d)_{V_r}$. Theorem \ref{heq-1} shows that $\MTU(d)_{V_r}$ is homotopy equivalent to $\MTU(d-r)$. It follows that: \[B\mathcal{C}^{\theta_{\C,r}}_{2d}\cong \Omega^{\infty+2d-1} \MTU(d-r)\]
\end{proof}
\begin{remark}
	The indexing of our spectra differs from the indexing in \cite{galatius}, leading to the $2d$ in the shift of the loop space $\Omega^{\infty+2d-1} \MTU(d-r)$.
\end{remark}

\begin{corollary}\label{oddiso}
	There is an isomorphism:
	$\pi_{2d+2r-1}(\MTU(d))\cong \pi_0 B\mathcal{C}^{\theta_{\C,r}}_{2(d+r)}$ 
\end{corollary}

The above corollary classifies all odd homotopy groups of $\MTU(d)$. (The lower homotopy groups of $\MTU(d)$ are classified by Lemma \ref{iso}).

\begin{remark}
	We note that we are not truly finding sections on $M^{2d-1}$ but on $M^{2d-1}\times I$ and our cobordisms are equivalences of these cylinders not of the lower dimensional manifold.
\end{remark}

\subsection{The odd dimensional complex section cobordism category}

We define a complex cobordism theory with cobordisms of dimension $2d+1$ and objects of dimension $2d$. This is a specific case of Definition \ref{cobcat} with structure $V_{\C,r}(U_{\C,d+1})\ra BO(2d+2)$.

\begin{definition}
	The cobordism category $\mathcal{C}^{U}_{2d+1,2n+2d+1}$ has as objects, $2d$ dimensional manifolds without boundary $M\subset\ (-1,1)^{2n+2d} \subset \R^{2n+2d}$ with $M$ closed in $\R^{2n+2d}$ along with a chosen lift of the map $\xi:M\ra G(2d,2n)\ra G(2d+2,2n)$ to $V_{\C,r}(U_{\C,d+1,n})$. The morphisms from $M_0\ra M_1$ are the disjoint union of the identity morphism along with pairs $(W,a)$ with the following properties. The real number $a\in(0,\infty)$ and $W\sb(-1,1)^{2n+2d}\times \R \subset \R^{2n+2d} \times \R $ is a manifold of dimension $2d+1$, which is closed in $\R^{2n+2d+1}$ and such that for some $\e>0$:
	\[W \cap (\R^{2n+2d}\times (-\infty,\e)) =M_0\times (-\infty,\e)\]
	\[W \cap (\R^{2n+2d}\times (a-\e,\infty)) =M_1\times (a-\e,\infty)\]
	Additionally $W$ is equipped with the data of a chosen lift of the classifying map $W\ra G(2d+1,2n)\ra G(2d+2,2n)$ to  $V_{\C,r}(U_{\C,d+1,n})$ which is compatible with the structures on the cobordant manifolds.
\end{definition}

For more details on the definition see \cite{galatius2} and \cite{bokstedt2}.
We typically consider the limit as $n\ra \infty$ with and abbreviate the notation for this category as $\mathcal{C}^{U,r}_{2d+1}$. Unraveling the definition, objects are manifolds $M^{2d}$ with complex structure and $r$ linearly independent complex sections on $TM\oplus \R^2$. Cobordisms $W$ are $2d+1$ dimensional manifolds with a complex structure and $r$ linearly independent complex sections on $TW \oplus \R$, compatible with the structure on the cobordant manifolds. We can define even dimensional complex section cobordism in a similar way to the odd case. 

\begin{definition}\label{evencsdef}
	Two $2d$ dimensional manifolds $M$ and $N$ with almost complex structure and $r$ linearly independent complex sections on $TM\oplus \R^2$ and $TN\oplus \R^2$ are defined to be complex section cobordant, if there is a cobordism $W$ with boundary $M\cup \overline{N}$ such that $TW\oplus \R$ has complex structure and $r$ linearly independent complex sections compatible with the structures on $TM\oplus \R^2$ and $TN\oplus \R^2$.
\end{definition}

Once again, this relation is equivalent to the existence of a morphism in $\mathcal{C}^{U,r}_{2d+1}$.

\begin{proposition}\label{eveneqrel}
	The relation defined in Definition \ref{evencsdef} is an equivalence relation.
\end{proposition}

\begin{theorem}
	The connected components of $B\mathcal{C}^{U,r}_{2d+1}$ are the equivalence classes of $2d$ dimensional manifolds under complex section cobordism.
\end{theorem}

The proofs of the proposition and theorem above are identical to the odd case. 

\begin{theorem}
	There is a homotopy equivalence
	\[B\mathcal{C}^{U,r}_{2d+1} \ra \Omega^{\infty+2d} \MTU(d+1-r)\]
	
\end{theorem}
\begin{proof}
	In \cite{galatius}, it was shown that there is a weak homotopy equivalence \[B\mathcal{C}^{U,r}_{2d+1}\ra \Omega^{\infty+2d} \theta^*_{\C,r} \mathbf{MTO}(2d+2)\]
	
	Since $\theta^*_{\C,r} \mathbf{MTO}(2d+2) \cong \MTU(d+1)_{V_r}\cong \MTU(d-r+1)$ as in the even case, the homotopy equivalence follows.
\end{proof}
\begin{corollary}\label{eveniso}
	There is an isomorphism.
	\[\pi_0 B\mathcal{C}^{U,r}_{2d+1} \cong \pi_{2d} (\MTU(d+1-r))\]
\end{corollary}
This corollary implies that $\pi_{2d} (\MTU(d+1-r))$ is the even complex section cobordism group. Note that objects of the $2d$ dimensional complex section cobordism group are manifolds $M^{2d}$ with complex structures and $r$ linearly independent complex sections on $TM\oplus\R^2$. However, by looking instead at the morphisms of the even complex section cobordism group, we observe that the elements of the group $\pi_1(B\mathcal{C}^{\theta_{\C,r}}_{2d})$ can be represented by $2d$ dimensional almost complex manifolds with $r$ linearly independent complex tangent sections. This is an immediate consequence of Theorem 3.5 in \cite{bokstedt2} and Proposition \ref{oddeqrel}. By using the established homotopy equivalences in Corollary \ref{oddiso}, we arrive at the following isomorphism:
\begin{proposition}
	There is an isomorphism
	\[\pi_1(B\mathcal{C}^{\theta_{\C,r}}_{2d}) \cong \pi_{2d} \MTU(d-r)\]
	Moreover every class in $\pi_{2d} \MTU(d-r)$, the even dimensional complex cobordism group, has a representative which is a $2d$ dimensional manifold with $r$ linearly independent complex sections on the tangent bundle itself.
\end{proposition}
This group is the main object of study in the next section. There is a geometric description of this isomorphism. Suppose we have an (almost complex) manifold $M^{2d}$ embedded in $\R^{2d+2n}$ equipped with $r$ linearly independent complex sections. Then we get a Gauss map $M^{2d} \ra G(2d,2n)$ which lifts as below. 

\begin{center}
	\begin{tikzpicture}
		
		\node (v1) at (-3,0) {$M^{2d}$};
		\node (v3) at (2,0) {$G(2d,2n)$};
		\draw [->] (v1) edge (v3);
		\node (v4) at (2,1.3) {$G_\C(d,n)$};
		\node (v5) at (2,2.6) {$V_{\C,r}(U_{\C,d,n})$};
		\draw [->] (v1) edge (v4);
		\draw [->] (v1) edge (v5);
		\draw [->] (v5) edge (v4);
		\draw [->] (v4) edge (v3);
	\end{tikzpicture}
\end{center}

If $\nu$ is the normal bundle of the embedding $M^{2d}\ra \R^{2d+2n}$, then we can construct the following commutative square.

\begin{center}
	\begin{tikzpicture}
		
		\node (v1) at (-4,0) {$M^{2d}$};
		\node (v3) at (-1,1.5) {$p_{V_{\C,r}}^* U^\perp_{\C,d,n}$};
		\node (v4) at (2,0) {$G_\C(d,n)$};
		\node (v5) at (-1,0) {$V_{\C,r}(U_{\C,d,n})$};
		\node (v6) at (2,1.5) {$U_{\C,d,n}^\perp $};
		\node (v2) at (-4,1.5) {$\nu$};
		\draw [->] (v2) edge (v1);
		\draw [->] (v2) edge (v3);
		\draw [->] (v3) edge (v6);
		\draw [->] (v6) edge (v4);
		\draw [->] (v5) edge (v4);
		\draw [->] (v6) edge (v4);
		\draw [->] (v1) edge (v5);
		\draw [->] (v3) edge (v5);
	\end{tikzpicture}
\end{center}

If we add one point to $\R^{2d+2n}$, we can consider an embedding $M^{2d}\ra S^{2d+2n}$. A small tubular neighborhood of $M$ is diffeomorphic to the total space of $\nu$. If we collapse outside this tubular neighborhood we get the Pontryagin-Thom map $ S^{2d+2n} \ra Th(\nu)$. Then, the above bundle maps give the composition \[ S^{2d+2n} \ra Th(\nu) \ra Th(p_{V_{\C,r}}^* (U^\perp_{\C,d,n})) \]
This composition represents $M^{2d}\in \pi_{2d}(\MTU_{V_r}(d))\cong \pi_{2d}(\MTU(d-r))$.

\section{The cobordism obstruction}
In the previous section, we reduced the study of the complex section cobordism groups to the study of the homotopy groups of $\MTU(d)$ using Theorems \ref{oddiso} and \ref{eveniso}. We had previously shown that there is a homomorphism: \[\pi_{q}(\MTU(d)) \ra \pi_q(\MU)\cong \Omega^U_{q}\]

When $q\leq 2d$, we showed that this map is an isomorphism. When $q>2d$, it is the forgetful map, which forgets the structure of the complex sections in $\pi_q(\MTU(d))$. To understand complex section cobordism, it is sufficient to understand the kernel and image of this map. The kernel represents the possible distinct structures of linearly independent complex sections on a manifold. The image will be the set of almost complex manifolds admitting a structure of $r$ complex sections. The rest of this section will focus on computing these groups using the cofibration long exact sequence in homotopy. We consider the following segment:
\[	\begin{array}{c}
	...\ra \pi_{2d+1}(\MU)\ra \pi_{2d+1}(\overline{\MTU}(d-r))\ra \pi_{2d}(\MTU(d-r)) \ra \\ \pi_{2d}(\MU) \xra{\gamma^r} \pi_{2d}(\overline{\MTU}(d-r))\ra \pi_{2d-1}(\MTU(d-r)) \ra \pi_{2d-1}(\MU) \ra ... 
\end{array} \]
Lemma \ref{iso} tells us that $\pi_{2d+1}(\MU)\cong\Omega_{2d+1}^U=0$ and $\pi_{2d-1}(\MU)\cong\Omega_{2d-1}^U=0$, leaving a five term exact sequence. Moreover, we know that $\pi_{2d}(\MU)\cong \Omega_{2d}^U$. So, the above exact sequence reduces to 
\[	\begin{array}{c}
0\ra \pi_{2d+1}(\overline{\MTU}(d-r))\ra \pi_{2d}(\MTU(d-r)) \ra  \Omega_{2d}^U\\ \xra{\gamma^r} \pi_{2d}(\overline{\MTU}(d-r))\ra \pi_{2d-1}(\MTU(d-r)) \ra 0 
\end{array} \]

For the rest of this section let $i_{d,r}$ be the rank of the group $H^{2d}(\MTU(d-r);\Q)$ and $j_{d,r}$ be the rank of $H^{2d}(\overline{\MTU}(d-r);\Q)$. A classical result of homotopy theory states that $H^*(X;\Q)\cong \pi_*(X)\otimes \Q$ for spectrum $X$ of finite type. Thus, we determine the ranks of the homotopy groups of $\MTU(d)$ and $\overline{\MTU}(d-r)$ by using Propositions \ref{cohMTU} and \ref{cohMTUbar}. Specifically, $i_{d,r}=\tn{Rank}(\pi_{2d}(\MTU(d-r)))=\tn{Rank}(\Ker(\gamma^r))$. Since the rank of $\pi_{q}(\overline{\MTU}(d-r))\otimes \Q$ is 0 for odd $q$, then $\pi_{q}(\overline{\MTU}(d-r))$ is a finite group for $q$ odd.

It follows from the above sequence that the map $\pi_{2d+1}(\overline{\MTU}(d-r))\ra \pi_{2d}(\MTU(d-r))$ is an injection. So, the forgetful map $\pi_{2d}(\MTU(d-r)) \ra \Omega_{2d}^U$ must have kernel $\pi_{2d+1}(\overline{\MTU}(d-r))$. Moreover, since $\Omega_{2d}^U$ is free abelian, the kernel must be the entire torsion of $\pi_{2d}(\MTU(d-r))$. Thus the group, $\pi_{2d}(\MTU(d-r))$ splits as  $\pi_{2d+1}(\overline{\MTU}(d-r))\oplus \Z^{\oplus i_{d,r}}$. Corollary \ref{stabcor} shows there is an isomorphism of homotopy groups $\pi_{2d+1}(\overline{\MTU}(d-r))\cong \pi_{2d+1}(\MTU(d+1,r+1))$. This proves the following theorem.
\begin{theorem}\label{splitting}
	There is an isomorphism
	\[\pi_{2d}(\MTU(d-r))\cong \pi_{2d+1}(\MTU(d+1,r+1)) \oplus \Z^{\oplus i_{d,r}}\]
	where $i_{d,r}$ is the rank of $\Z[c_1,...,c_{d-r}]$ in degree $2d$. 
\end{theorem}
Corollary \ref{eveniso} shows that $\pi_{2d}(\MTU(d-r))$ is the complex section cobordism group of $2d$ dimensional manifolds $M$ with $r$ linearly independent sections. We conclude the following result:
\begin{corollary}\label{T2body}
	Let $[M^{2d}]\in \Omega_{2d}^U$, be such that $M^{2d}$ is equipped with $r$ linearly independent complex sections on $TM$. Then, there are only finitely many pairs, up to complex section cobordism, $(N^{2d},s)$ where the manifold $N^{2d}\in [M^{2d}]$ and the map $s:N\ra V_{\C,r}(TN)$ is the structure of $r$ linearly independent complex sections on $N^{2d}$. Moreover, such pairs are indexed by the group $\pi_{2d+1}(\MTU(d+1,r+1))$.
\end{corollary}
A similar result is also true for odd dimensional manifolds. The above corollary and following proposition provide the proofs of Theorems \ref{T2} and \ref{T3}.
\begin{proposition}\label{T3body}
	The odd dimensional complex section cobordism group is finite.
\end{proposition}
\begin{proof}
	The odd dimensional complex section cobordism group is $\pi_{2d+2r-1}(\MTU(d))$.
	
	Since $H^{2d+2r-1}(\MTU(d);\Q)=0$, $\pi_{2d+2r-1}(\MTU(d))\otimes \Q =0$. Thus the group $\pi_{2d+2r-1}(\MTU(d))$ is finite.
\end{proof}

Let $\gamma^r:\Omega_{2d}^{U} \ra \pi_{2d}(\overline{\MTU}(d-r))$ be the induced map in homotopy from $\MU\ra \overline{\MTU}(d-r)$. By the homotopy exact sequence, a manifold $M\in \Omega_{2d}^U$ lifts to the group $\pi_{2d}(\MTU(d-r))$ if and only if $\gamma^r(M)=0$. We conclude that:

\begin{theorem}\label{obstruction}
	A manifold $M$ is complex cobordant to a manifold equipped with $r$ linearly independent complex sections on $TM$ if and only if $\gamma^r(M)=0$.
\end{theorem}

Importantly, the vanishing of $\gamma^r(M)$ is a sufficient, not just a necessary condition for finding $r$ linearly independent complex sections. To describe the obstruction map $\gamma^r$, we use the characteristic classes $s_\omega$. These classes are defined in \cite{Thom} and \cite{milnor}. Let $\omega=\{i_1\geq i_2 \geq...\geq i_k \neq 0\}$ be a partition of $d$ and define the length of $\omega$ to be $l(\omega)=k$. Let $f_\omega(t_1,...,t_k)$ be the least symmetric polynomial in variables $t_1,...,t_d$ with $t^{i_1}_1...t^{i_k}_k$ as a summand. We can write $f_\omega(t_1,...,t_k)$ as a function of the elementary symmetric polynomials $\sigma_1(t_1,...,t_d)$,..., $\sigma_d(t_1,...,t_d)$ so that $f_\omega(t_1,...,t_k)=s_\omega(\sigma_1,...\sigma_d)$ where $s_\omega$ is a polynomial. Let $M^{2d}$ be an almost complex manifold. Then, the characteristic numbers $s_\omega(M^{2d})$ are defined as \[s_\omega(M^{2d})=\k{s_\omega(c_1(TM),...,c_d(TM)), [M^{2d}]}\]
A special example is $s_{1,1,..,1}(M^{2d})$. The least symmetric polynomial in variables $t_1,...,t_d$ with $t_1...t_d$ as a summand is the elementary symmetric polynomial $t_1t_2...t_d$ itself. So the polynomial $s_{1,1,...,1}(c_1,...,c_d)=c_d$ and $s_{1,1,..,1}(M^{2d})=\chi(M^{2d})$. We know by the theorem of Hopf \cite{hopf} that this is the only obstruction for $r=1$ vector field. It will be helpful to recall the following property, see \cite{milnor} or \cite{Thom}.

\begin{lemma}\label{alglem}
	If $c_k=0$ for all $k\geq N$, then $s_\omega(c_1,...,c_d)=0$ for all $\omega$ such that $l(\omega)\geq N$.
\end{lemma}
\begin{proof}
	Suppose $c_k=0$ for all $k\geq N$. Let $\omega$ have length $N$. The polynomial $f_\omega(t_1,...,t_d)$ consists of monomials of the form $t^{i_1}_1...t^{i_N}_N$ where $\omega=\{i_1\geq ...\geq i_N \neq 0\}$ up to permutation of the variables $t_1,..., t_d$.  Since, by assumption, $c_N=\sigma_N(t_1,...,t_d)=0$, we can write $t_1...t_N=-\sum t_{j_1}...t_{j_N}$ where $\{j_1,..., j_N\}\neq \{1,...,N\}$. In particular: \[t^{i_1}_1...t^{i_N}_N= -t^{i_1-1}_1...t^{i_N-1}_N\sum t_{j_1}...t_{j_N}\]
	Notice that each monomial in this sum has more than $N$ distinct $t_j$. Thus we may repeat the process (along with symmetry), using $c_{M}=0$ for $M\geq N$, to increase the number of distinct $t_j$ in the polynomial up to $d$. Thus $f_\omega(t_1,...,t_d)=a t_1...t_d=ac_d=0$ for some integer $a$, if $c_k=0$ for all $k\geq N$. We conclude that $s_\omega(c_1,...,c_d)=0$ for all $\omega$ such that $l(\omega)\geq N$. 
\end{proof}

We now establish the rational obstruction, Theorem \ref{ratobsintro}, in terms of characteristic numbers.

\begin{theorem}\label{ratobs}
	Let $M^{2d}$ be a $d$-dimensional almost-complex manifold with $r<d$. Let \[\gamma^r_\Q:=\gamma^r \otimes \Q: \Omega_{2d}^U\otimes \Q \ra \pi_{2d}(\overline{\MTU}(d-r))\otimes \Q\]
	Then $\gamma^r_\Q(M^{2d})=0$ if and only if $s_\omega(M^{2d})=0$ for all $\omega$ of length greater than $d-r$.
\end{theorem}

\begin{corollary}
	Let $M^{2d}$ be a $d$-dimensional almost-complex manifold. Then there exists constant $c$ such that the cobordism class $c[M^{2d}]$ contains a manifold $N^{2d}$ with $r$ complex sections on $TN$ if and only if $s_\omega(M^{2d})=0$ for all $\omega$ of length greater than $d-r$.
\end{corollary}
We need the following theorem of Stong to determine $\gamma^r_\Q$.  \cite[117]{Stong} (This theorem is true integrally but we will not use this fact here.)
\begin{theorem}\label{Stong}
	The set $\{s_{k_1,...,k_r}\mid k_1+...+k_r=d\}$ is a basis of $\Hom_\Q(\Omega^U_{2d}\otimes \Q, \Q)$.
\end{theorem}

\begin{proof}[Proof of Theorem \ref{ratobs}]
	We start by proving the forward direction, which is a consequence of known results. Suppose $\gamma_r(M^{2d})=0$. By Theorem \ref{obstruction}, there is a cobordant manifold $N^{2d}$ such that $TN^{2d}$ has $r$ linearly independent complex sections. Thus, $TN^{2d}$ splits into $E\oplus \C^{r}$ where $E$ has complex dimension $d-r$. Thus $c_{k}(TN^{2d})=0$ for $k\geq d-r+1$. By Lemma \ref{alglem}, $s_\omega(N^{2d})=0$ for $\omega$ of length greater than $d-r$. Since cobordant manifolds have the same Chern numbers, this direction is proven.

	We know that $\gamma^r_\Q: \Omega^U_{2d}\otimes \Q \ra \pi_{2d}(\overline{\MTU}(d-r))\otimes \Q \cong \Q^{j_{d,r}}$ has rank $j_{d,r}$ and is surjective. Each summand of this map can be written as a linear combination of $s_\omega$ by Theorem \ref{Stong}. Let $S'\sb \Hom(\Omega^U_{2d}\otimes \Q,\Q)$ be the span of all $s_\omega$ such that $l(\omega)\geq d-r+1$. Let $S \sb \Omega^U_{2d}\otimes \Q$ be the vector space of all elements $x$ such that $r(x)=0 $ for every element $r\in S'$. Let $[M^{2d}]\notin S$. Then, for some $\omega$ with $l(\omega)\geq d-r+1$, $s_\omega([M^{2d}])\neq 0$. Moreover for any integer $n$, $s_\omega(n[M^{2d}])\neq 0$.
	By Lemma \ref{alglem}, $c_k(n[M^{2d}])\neq 0$ for some $k\geq d-r+1$. So, no manifold in the class $n[M^{2d}]$ may have $r$ sections for any integer $n$, because it has non-zero characteristic class in dimension greater than $d-r$. By Theorem \ref{obstruction}, every element $[M^{2d}]$ not in $S$ must have $\gamma^r_\Q([M^{2d}])\neq 0$. We conclude that $\Ker(\gamma^r_\Q)\sb S$. 
	
	It remains to show that the space $S$ has rank $i_{d,r}$. We recall that $i_{d,r}$ is the rank of the vector space $H^*(\MTU(d-r);\Q)\cong\Q[c_1,...,c_{d-r}]$ in degree $2d$. This rank is equivalent to partitions of $d$ by integers less than or equal to $d-r$. The rank of $S$ is the rank of $\Omega^U_{2d}\otimes \Q$ minus the rank of $S'$. Since the rank of $\Omega^U_{2d}\otimes \Q$ is the number of partitions of $d$ and the rank of $S'$ is the number of partitions of $d$ with length greater than $d-r$, we conclude that the rank of $S$ is the number of partitions of length less than or equal to $d-r$. Thus, both $S$ and $\Ker(\gamma^r_\Q)$ have rank $i_{d,r}$. Since $\Ker(\gamma^r_\Q)$ is a vector space, $\Ker(\gamma^r_\Q)=S$.
\end{proof}

It will be useful to discuss the multiplicative structure of the complex cobordism ring. The following result is used to verify whether certain manifolds are multiplicative generators. \cite[128]{Stong}
\begin{theorem} \label{gen}
	The complex cobordism ring $\Omega^U$ is isomorphic to the polynomial ring $ \Z[b_1,b_2,...]$ with generators $b_i$ in dimension $2i$. If $i\neq p^q-1$ for any prime $p$, then a manifold $M^{2i}$ can be taken to be the generator if and only if $s_i([M^{2i}])=\pm1$. If $i= p^q-1$ for some prime $p$, then $M^{2i}$ can be taken to be the generator if and only if $s_i([M^{2i}])=\pm p$.
\end{theorem}
As a corollary, when we look at the rational complex cobordism ring, the generators are identified by the following theorem:
\begin{corollary} \label{ratgen}
	The rational complex cobordism ring is $\Omega^U\otimes \Q\cong \Q[b_1,b_2,...]$ with generators $b_i$ in dimension $2i$. A manifold $M^{2i}\in \Omega_{2i}\otimes \Q$ can be taken to be the multiplicative generator if $s_i([M^{2i}])\neq 0$
\end{corollary}
\begin{proof}
	Let $k_i=1$ if $i\neq p^q-1$ for all primes $p$ and $k_i=p$ if $i\neq p^q-1$ for some prime $p$. Suppose $M^{2i}$ is such that $s_i([M^{2i}])=k\in \Q$. Choose sufficiently large integer $n$ such that $k_i$ divides $nk$. Choose any $\tilde{M}^{2i}$ which is a representative of a generator in $\Omega^U$. Since $s_i(\tilde{M}^{2i})=k_i$, there exists $m$ such that $s_i(m\tilde{M}^{2i}-n[M^{2i}])=0$. Thus $m\tilde{M}^{2i}-n[M^{2i}]$ is represented by a decomposable manifold by Corollary 16.5 \cite{milnor}. Thus $n[M^{2i}]$ is a generator, since it differs from a generator by decomposables, and $[M^{2i}]$ is also a rational generator.
\end{proof}
We finish by proving Theorem \ref{T4}
\begin{theorem}\label{T4body}
	Let $r< d$. There exists a manifold $M^{2d}$ in $\Omega_{2d}^U$ which can be equipped with $r$ linearly independent complex sections on $TM$ and whose image in $\Omega_{2d}^U\otimes \Q$ is a multiplicative generator.
\end{theorem}
\begin{proof}
	By \cite{milnor} Theorem 16.7, the maps $s_\omega$ span the $p(d)$-dimensional vector space $\Hom_\Q(\Omega_{2d}^U\otimes \Q, \Q)$. Since there are $p(d)$ characteristic classes, they must also be linearly independent. Moreover, for every $\omega$, we can find a dual object $M^{2d}_\omega\in \Omega_{2d}^U\otimes \Q$ such that $s_\omega(M^{2d}_\omega)=1$ and all other characteristic numbers are 0. If we choose $\omega$ to be the partition $d$, then we get a rational generator $M^{2d}_d$ by Corollary \ref{ratgen}. We can choose integer $c$ such that $M^{2d}:=cM^{2d}_d\in \Omega_{2d}^U$. Then $s_\omega(M^{2d})=0$ for $\omega\neq d$ and Theorem \ref{ratobs} shows that some manifold $\tilde{N}^{2d}\in[M^{2d}]$ has $r$ linearly independent complex sections.
\end{proof}
\bibliography{ref}{}
\bibliographystyle{plain}
\end{document}